\documentclass[12pt]{amsart}
\usepackage{epsfig,color}

\makeatother

\theoremstyle{definition}

\def\fnum{equation}
\newtheorem{Thm}[\fnum]{Theorem}
\newtheorem{Cor}[\fnum]{Corollary}

\newtheorem{Lem}[\fnum]{Lemma}

\numberwithin{equation}{section}

\newcommand{\Vol}{{\text{Vol}}}
\newcommand{\V}{{\text{V}}}

\newcommand{\Ric}{{\text{Ric}}}
\newcommand{\Ker}{{\text{Ker}}}

\newcommand{\cK}{{\mathcal{K}}}

\def\RR{{\bold R}}

\newcommand{\cA}{{\mathcal{A}}}

\newcommand{\cH}{{\mathcal{H}}}

\newcommand{\cP}{{\mathcal{P}}}

\newcommand{\eqr}[1]{(\ref{#1})}

\title[Optimal bounds for ancient caloric functions]{Optimal bounds for ancient caloric functions}

\author[]{Tobias Holck Colding}%
\address{MIT, Dept. of Math.\\
77 Massachusetts Avenue, Cambridge, MA 02139-4307.}
\author[]{William P. Minicozzi II}%
 
\thanks{The  authors
were partially supported by NSF Grants DMS 1812142 and DMS 1707270.}

\email{colding@math.mit.edu  and minicozz@math.mit.edu}

\begin{document}

\maketitle

\begin{abstract}
For any manifold with polynomial volume growth, we show:  The dimension of the space of ancient caloric functions with polynomial growth is bounded by the degree of growth times the dimension of harmonic functions with the same  growth.  As a consequence, we get a sharp bound for the dimension of ancient caloric functions on any space where Yau's 1974 conjecture about polynomial growth harmonic functions holds.
\end{abstract}

\section{Introduction}

Given a complete  manifold $M$ and a    constant $d$,  $\cH_d(M)$ is the linear space of harmonic functions of polynomial growth at most $d$.  Namely, 
  $u \in \cH_d (M)$ if $\Delta u = 0$ and for some $p \in M$ and a constant $C_u$ depending on $u$
\begin{align}
	\sup_{B_R(p)} |u| \leq C_u\, (1 + R)^d {\text{ for all }} R \, .
\end{align}

In 1974, S.T. Yau conjectured that $\cH_d (M)$ is finite dimensional for each $d$ when $\Ric_M\geq 0$.  
  The conjecture  was settled in \cite{CM2}; see \cite{CM1}--\cite{CM5} for more results.\footnote{For Yau's 1974 conjecture see:
   page $117$ in \cite{Ya2}, problem $48$ in \cite{Ya3}, Conjecture $2.5$  in \cite{Sc}, \cite{Ka}, \cite{Kz}, \cite{DF}, Conjecture $1$ in \cite{Li1}, and problem (1) in \cite{LiTa}, amongst others.}  In fact, \cite{CM2}--\cite{CM4} proved finite dimensionality under much weaker assumptions of:
\begin{enumerate}
\item  A volume doubling bound.
\item A scale-invariant Poincar\'e inequality or meanvalue inequality.
\end{enumerate}

The natural parabolic generalization   is a polynomial growth ancient solution of the heat equation.   A solution of the heat equation is often called a caloric function.  Ancient solutions are ones that are defined for all negative $t$ - these are the solutions that arise in a blow up analysis.
Given $d>0$,  $u \in \cP_d (M)$ if $u$ is ancient,  $\partial_t u = \Delta \, u$ and for some $p \in M$ and a constant $C_u$  
\begin{align}
	\sup_{ B_R(p) \times [-R^2,0]}\, |u| \leq C_u \, \left( 1 + R \right)^d {\text{ for all $R$}} \, .
\end{align}
On $\RR^n$,  $ \cP_d $ is the classical space of caloric polynomials that generalize the Hermite polynomials; see \cite{N}, \cite{E1}, \cite{E2}.
 More generally, the spaces $\cP_d(M)$ play a fundamental role in geometric flows, see \cite{CM6}--\cite{CM8}.   They were studied by Calle in her 2006 thesis, \cite{Ca1}, \cite{Ca2},   in the context of mean curvature flow.

A manifold has polynomial volume growth if there are   constants $C$ and $d_V$ so that
$\Vol (B_R(p)) \leq C \, (1+R)^{d_V}$ for some $p \in M$, all $R>0$.\footnote{A volume doubling space with doubling constant $C_D$ has polynomial volume growth of degree ${\log_2 C_D}$.}
Our main result is the following sharp inequality:

\begin{Thm}   \label{t:main2}
If $M$ has polynomial volume growth and $k$ is a nonnegative integer, then
\begin{align}	\label{e:0p4}
\dim \cP_{2k }(M)\leq \sum_{i=0}^k\,\dim \cH_{2i }(M)  \, .
\end{align}
\end{Thm}

The inequality \eqr{e:0p4}   is an equality on $\RR^n$ (see Corollary \ref{c:c218} below).
Since $\cH_{d_1}\subset \cH_{d_2}$ for $d_1\leq d_2$,  Theorem \ref{t:main2} implies:

\begin{Cor}	\label{t:CMPd1}
If $M$ has polynomial volume growth,  then for all $k \geq 1$ 
\begin{align}
	\dim \cP_{2k}(M) \leq (k+1) \, \dim \cH_{2k}(M) \, .
\end{align}
\end{Cor}

 
Combining this
with the bound $   \dim \cH_d (M) \leq C \, d^{n-1}$ when  $\Ric_{M^n} \geq 0$  from \cite{CM3} gives:
 
\begin{Cor}	\label{t:CMPd}
There exists $C=C(n)$ so that if   $\Ric_{M^n} \geq 0$, then for $d\geq 1$
\begin{align}	\label{e:cmpd}
	\dim \, \cP_d (M) \leq C \, d^n \, .
\end{align}
\end{Cor}

The exponent $n$ in \eqr{e:cmpd}   is  sharp:  There is a constant $c$ depending on $n$ so that for $d\geq 1$
\begin{align}	\label{e:cpdrn}
	c^{-1} \, d^n \leq \dim \cP_d (\RR^n) \leq c \, d^n \, .
\end{align}
Recently, Lin and Zhang, 
\cite{LZ},  proved   very interesting  related results,  adapting the methods of \cite{CM2}--\cite{CM4} to get the  bound $d^{n+1}$.

Using parabolic gradient estimates of Li-Yau, \cite{LiY}, and Souplet-Zhang,  \cite{SoZ}, one can show that
  if $d < 2$ and $\Ric \geq 0$, then $\cP_d(M) = \cH_d(M)$ consists only of harmonic functions of polynomial growth.  In particular,  $\cP_d (M) = \{ {\text{Constant functions}} \}$ for  $d<1$ and,
 moreover, 
$\dim \cP_1 (M) \leq n+1$, by Li and Tam, \cite{LiTa}, with equality if and only if $M = \RR^n$ by \cite{ChCM}.   

The exponent 
  $n-1$ is also sharp in the bound for $\dim \cH_d$ when $\Ric_{M^n} \geq 0$.  However, as in Weyl's
asymptotic formula, the coefficient of $d^{n-1}$ can be
related to the volume, \cite{CM3}:
\begin{equation}    \label{e:cm15a}
    \dim \cH_d (M)\leq C_n \, \V_M \, d^{n-1} + o (d^{n-1}) \,  .
\end{equation}
\begin{itemize}  
 \item
  $\V_M$ is the ``asymptotic volume ratio'' $\lim_{r\to \infty} \,
  \Vol (B_r)/ r^n$. \item $o(d^{n-1})$ is a function of $d$ with
  $\lim_{d\to \infty} \, o(d^{n-1})/d^{n-1} = 0$.
  \end{itemize}
Combining \eqr{e:cm15a} with Corollary \ref{t:CMPd1} gives $\dim \, \cP_d (M) \leq C_n \, \V_M \, d^{n} + o (d^{n})$ when $\Ric_{M^n} \geq 0$.

An interesting  feature of these dimension estimates is that they
follow from ``rough'' properties of $M$ and are therefore
surprisingly stable under perturbation. For instance, 
\cite{CM4}  proves finite dimensionality of $\cH_d$  for manifolds
with a volume doubling and a Poincar\'e inequality, so we also get finite dimensionality for $\cP_d$ on these spaces. 
Unlike
a Ricci curvature bound, these properties are stable under
bi--Lipschitz transformations (cf. \cite{MS}). Moreover, these properties  make sense also for discrete spaces, vastly extending the theory and methods out of the continuous world.  
Recently Kleiner, \cite{K}, (see also Shalom-Tao, \cite{ST}, \cite{T1}, \cite{T2}) used, in part, this in his new proof of an important and foundational result in geometric group theory, originally due to Gromov, \cite{G}.    We expect that the proof of Theorem \ref{t:main2}  extends to many discrete spaces, allowing a wide range of applications.

\section{Ancient solutions of the heat equation}

The next lemma gives a reverse Poincar\'e inequality for the heat equation (cf. \cite{M}).
 
\begin{Lem}	\label{l:RPh}
There is a universal constant $c$ so that if $u_t=\Delta\,u$, then 
\begin{align}
	   r^2 \, \int_{ B_{ \frac{r}{10} } \times [ - \frac{r^2}{100} , 0]   }|\nabla u|^2 + r^4 \, \int_{ B_{ \frac{r}{10} } \times [ - \frac{r^2}{100} , 0]   } u_t^2 &\leq c \,  \int_{B_r \times [-r^2 , 0]} u^2   \, .
\end{align}
 \end{Lem}

\begin{proof}
Let $Q_R$ denote $B_R \times [-R^2,0]$ and $\psi$ be a cutoff function on $M$.  Since $u_t = \Delta\, u$, integration by parts and the absorbing inequality $4ab \leq a^2 + 4b^2$ give
\begin{align}
	\partial_t \, \int u^2 \, \psi^2 &= 2 \, \int u\, \psi^2 \, \Delta u = - 2 \int |\nabla u|^2 \, \psi^2 - 4 \int u\, \psi\, \langle \nabla \psi , \nabla u \rangle \notag \\
		&\leq -   \int |\nabla u|^2 \, \psi^2  + 4 \int u^2 \, |\nabla \psi |^2 \, .
\end{align}
Integrating this  in time from $-R^2$ to $0$ gives
\begin{align}
	\int_{t=0} u^2\, \psi^2 - \int_{t=-R^2} u^2\, \psi^2 & \leq  \int_{-R^2}^0 \, \left(  -   \int |\nabla u|^2 \, \psi^2  + 4 \int u^2 \, |\nabla \psi |^2 \right) \, dt  \, .
\end{align}
In particular, we get 
\begin{align}
	\int_{-R^2}^0 \,    \int |\nabla u|^2 \, \psi^2 \, dt  \leq \int_{t=-R^2} u^2\, \psi^2 + 4 \,  \int_{-R^2}^0 \,   \int u^2 \, |\nabla \psi |^2  \, dt   \, .
\end{align}
Let $|\psi | \leq 1$   be one on $B_{R/2}$,  have support in $B_R$, and satisfy   $|\nabla \psi | \leq 2/R$, so we get
\begin{align}	\label{e:R1}
	\int_{Q_{ R/2}} |\nabla u|^2 \leq \int_{B_R \times \{ t=-R^2 \} } \, u^2 + \frac{16}{R^2} \, \int_{Q_R} u^2 \, .
\end{align}
Next, we argue similarly to get a bound on $u_t^2$.  Namely, differentiating, then integrating by parts and using that $u_t = \Delta\, u$ gives
\begin{align}	\label{e:e1p7}
	\partial_t \, \int |\nabla u|^2\, \psi^2 &= 2 \int \langle \nabla u , \nabla u_t \rangle\, \psi^2 = - 2 \int   u_t^2\, \psi^2 - 4 \int u_t\, \psi \, \langle \nabla u , \nabla \psi \rangle \notag \\
	&\leq  -   \int   u_t^2\, \psi^2 + 4 \int  | \nabla u|^2\, | \nabla \psi|^2 \, .
\end{align}
Integrating \eqr{e:e1p7} in time from $-R^2$ to $0$ gives
\begin{align}
	\int_{t=0}  |\nabla u|^2\, \psi^2 - \int_{t=-R^2}  |\nabla u|^2\, \psi^2 \leq \int_{-R^2}^0 \left(  -   \int   u_t^2\, \psi^2 + 4 \int  | \nabla u|^2\, | \nabla \psi|^2 \right) \, dt \, .
\end{align}
Letting $\psi$ be as above, we conclude that
\begin{align}	\label{e:R2}
	\int_{ Q_{ R/2} } u_t^2 \leq \frac{16}{R^2} \int_{Q_R} |\nabla u|^2 +  \int_{B_R \times \{ t=-R^2 \} }   |\nabla u|^2 \, .
\end{align}

Next,   choose some $r_1 \in [4r/5 , r]$ with 
\begin{align}	\label{e:atr1}
	\int_{ B_r \times \{ t = - r_1^2 \} } \, u^2 \leq \frac{25}{9\,r^2} \, \int_{-r^2}^0  \, \left( \int_{B_r} u^2 \right)  \, dt=\frac{25}{9\,r^2} \, \int_{Q_r} u^2 \, .
\end{align}
Applying \eqr{e:R1} with $R = r_1$ and using the bound \eqr{e:atr1} at $r_1$ gives
\begin{align}	\label{e:R1a}
	\int_{Q_{ \frac{2\,r}{5}}   } |\nabla u|^2 \leq   \int_{Q_{ \frac{r_1}{2}}   } |\nabla u|^2 \leq    \int_{B_{r_1} \times \{ t=-r_1^2 \} } \, u^2 + \frac{16}{r_1^2} \, \int_{Q_{r_1}} u^2  \leq \frac{20}{r_1^2} \int_{Q_r} u^2     \, .
\end{align}
For simplicity, $c$ is a constant independent of everything that can change from line to line.
It follows from \eqr{e:R1a} that there must exist some $\rho   \in [r/5 , 2r/5]$ so that 
\begin{align}	\label{e:R1b}
	\int_{B_{\frac{2\,r}{5}} \times \{ t = -\rho^2 \} } \, |\nabla u|^2 \leq \frac{25}{3r^2} \, \int_{-\frac{4\,r^2}{25}}^0  \left( \int_{B_{ \frac{2\,r}{5}}   } |\nabla u|^2 \right) \, dt
	=\frac{25}{3r^2} \, \int_{Q_{ \frac{2\,r}{5}}   } |\nabla u|^2 \leq \frac{c}{r^4} \, 
	 \int_{Q_r} u^2    \, .
\end{align}
Now applying \eqr{e:R2} with $R = \rho$ and using \eqr{e:R1a} and \eqr{e:R1b} gives 
\begin{align}
	  \int_{ Q_{ \rho/2} } u_t^2 &\leq \frac{16}{\rho^2} \int_{Q_{\rho}} |\nabla u|^2 +  \int_{B_{\rho} \times \{ t=-\rho^2 \} }   |\nabla u|^2  \leq  \frac{c}{r^4} \, 
	 \int_{Q_r} u^2   \, .
\end{align}
\end{proof}

\begin{Cor}	\label{c:tpolynomial}
If $\Vol (B_R) \leq C \, (1+R)^{d_V}$ and $u \in \cP_d(M)$, then $\partial_t^k u \equiv 0$ for $4k > 2d+d_V+2$.
\end{Cor}

\begin{proof}
Since the metric on $M$ is constant in time, $\partial_t - \Delta$ commutes with $\partial_t$ and, thus, $(\partial_t - \Delta) \partial^j_t u=0$  for every $j$.  
Let $Q_R$ denote $B_R \times [-R^2,0]$.   Applying 
Lemma \ref{l:RPh} to $u$ on $Q_r$ for some $r$, then to $u_t$ on $Q_{ \frac{r}{10} }$, etc., we get a constant $c_k$ depending just on $k$ so that
\begin{align}
	\int_{ Q_{ \frac{r}{10^k} }} \left| \partial^k_t u \right|^2 \leq \frac{c_k}{ r^{4k} } \, \int_{ Q_r } u^2  
	\leq \frac{c_k}{ r^{4k} } \, r^2 \, \Vol (B_r) \, \sup_{Q_r} u^2  \leq C\, c_k \, r^{2-4k} \, (1+r)^{2d+d_V}  \, .
\end{align}
 Since  $4k > 2d+d_V+2$, the right-hand side goes to zero as $r \to \infty$, giving the corollary.
\end{proof}

We will prove Corollary \ref{t:CMPd1} next, though it will eventually  be a corollary of Theorem \ref{t:main2}.

\begin{proof}[Proof of Corollary \ref{t:CMPd1}]
Choose an integer $m$ with $4m > 2k+d_V+2$.   
Corollary \ref{c:tpolynomial} gives that $\partial_t^m u=0$ for any $u \in \cP_{2k}(M)$.
  Thus, any $u \in \cP_{2k} (M)$  can be written as
\begin{align}	\label{e:decompou}
	u = p_0 + t \, p_1 + \dots + t^{m-1} \, p_{m-1} \, ,
\end{align}
where each $p_j$ is a function on $M$.    Moreover, using the growth bound 
$u \in \cP_{2k}(M)$ for $t$ large and $x$ fixed, we see that $p_j \equiv 0$ for any $j> k$.   (See   theorem $1.2$ in \cite{LZ} for a similar decomposition under more restrictive hypotheses and \cite{KoT} for a splitting result for ancient positive solutions on homogeneous spaces.) 

We will show next that the functions $p_j$ grow at most polynomially of degree $d$.   Fix distinct values $-1<t_1 < t_2 < \dots < t_k < t_{k+1}=0$.  We claim that the $k+1$-vectors
\begin{align}
	(1, t_i , t_i^2 , \dots , t_i^k)
\end{align}
 are linearly independent in $\RR^{k+1}$ for $i=1 , \dots , k+1$.  If this was not the case, then there would be some (non-trivial) $(a_0 , \dots , a_k) \in \RR^{k+1}$ that is orthogonal to all of them.
But this means that there would be $k+1$ distinct roots to the degree $k$ polynomial 
\begin{align}
	a_0 + a_1 t + \dots + a_k t^k \, ,
\end{align}
 which is impossible, and the claim follows.  Let $e_j \in \RR^{k+1}$ be the standard unit vectors. 
Since the 
$(1, t_i , t_i^2 , \dots , t_i^k)$'s span $\RR^{k+1}$, we can choose coefficients $b^j_i$ so that for each $j$
\begin{align}	\label{bijej}
	e_j  = \sum_i  b^j_i \, (1, t_i , t_i^2 , \dots , t_i^k) \, .
\end{align}
It follows from \eqr{e:decompou} and \eqr{bijej}  that
\begin{align}	\label{e:ply2k}
	p_j(x) = \sum_i b^j_i u(x , t_i) \, .
\end{align}
Since $u \in \cP_{2k} (M)$, \eqr{e:ply2k} implies that each $p_j$ is a linear combination of functions that grow polynomially of degree at most $2k$ and, thus,  $p_j$ grows polynomially of degree at most $2k$.

Since $u$ satisfies the heat equation, it follows that $\Delta p_{k} = 0$ and
\begin{align}	\label{e:deco1}
	\Delta p_j = (j+1) \, p_{j+1} \, .
\end{align}
Thus, we get a linear map $\Psi_0 : \cP_{2k} (M) \to \cH_{2k} (M)$ given by $\Psi_0 (u) = p_k$.  Let $\cK_0 = \Ker (\Psi_0)$.  It follows from this   that
\begin{align}	\label{e:cK01}
	\dim \cP_{2k}(M) \leq \dim \cK_0 + \dim \cH_{2k}(M)  \, .
\end{align}
If $u \in \cK_0$, then $p_k =0$ and  $\Delta p_{k-1} = 0$, so  we get a linear map  $\Psi_1 : \cK_0 \to \cH_{2k} (M)$ given by $\Psi_1 (u) = p_{k-1}$. Let $\cK_1$ be the kernel of $\Psi_1$ on $\cK_0$.  It follows as above that
\begin{align}	\label{e:cK11}
	  \dim  \cK_0  \leq \dim \cK_1  +   \dim \cH_{2k}(M) \, .
\end{align}
Repeating this $k+1$ times gives the theorem.
\end{proof}

\begin{Lem}	\label{l:strati}
If $u \in \cP_{2k}(M)$ can be written as $u=p_0(x) + t \, p_1(x) + \dots + t^k \, p_k(x)$, then  
\begin{align}		\label{e:strati}
	|p_j(x)| \leq C_j \, \left( 1 + |x|^{2(k-j)} \right) \, .
\end{align}
\end{Lem}

\begin{proof}
By assumption, there is a constant $C$ so that
\begin{align}	\label{e:defcpd1}
	|u(x,t)| \leq C \, \left( 1 + |t|^k + |x|^{2k} \right) \, .
\end{align} 
Following the proof of Corollary \ref{t:CMPd1},  fix $-1<t_1 < t_2 < \dots < t_k < t_{k+1}= - \frac{1}{2}$ and coefficients $b^j_i$ so that  \eqr{bijej}  holds for each $j$.
 Observe that  \eqr{bijej}  gives for each   $j$ 
\begin{align}
	\sum_i b_i^ju(x,R^2\, t_i) =  \sum_i   \sum_{\ell}  b_i^j \, p_{\ell}(x) \, R^{2j} \, t_i^{\ell} =   \sum_{\ell} \sum_i   b_i^j \, p_{\ell}(x) \, R^{2j} \, t_i^{\ell}  =R^{2j} \, p_j(x)  \, .
\end{align}
Thus, given $R> 2$ and $x \in B_R$,  we get  that
\begin{align}	\label{e:ply2k}
	\left| R^{2j} \, p_j(x) \right| &= \left|  \sum_i b^j_i u(x , R^2\, t_i) \right| \leq \max_{i,j} \, |b^j_i| \, \sum_i |u(x,R^2\, t_i)| \notag \\
	&\leq \tilde{C} \, \left( 1 + |x|^{2k} + \max_i |R^2 \, t_i |^{ k} \right) \leq 3 \, \tilde{C} \, R^{2k} \, .
\end{align}
 From this, we conclude that
 	$\sup_{B_R} \, |p_j| \leq 3 \, \tilde{C} \, R^{2k-2j}$.    
\end{proof}

 \begin{proof}
 (of Theorem \ref{t:main2}).  Following the proof Corollary \ref{t:CMPd1}, each $u \in  \cP_{2k}(M)$,  can be expanded as
$
 	u=p_0(x) + t \, p_1(x) + \dots + t^k \, p_k(x)$.	   By Lemma \ref{l:strati},    the linear map $\Psi_0 : \cP_{2k} (M) \to \cH_{2k} (M)$ given by $\Psi_0 (u) = p_k$ actually maps into $\cH_0 (M)$ and, thus, 
\begin{align}	\label{e:cK01}
	\dim \cP_{2k}(M) \leq \dim \cH_{0}(M) + \dim \Ker (\Psi_0)    \, .
\end{align}
Similary,  Lemma \ref{l:strati} implies that the map $\Psi_1$ maps the kernel of $\Psi_0$  to $\cH_2(M)$.  Applying this repeatedly gives the theorem.
\end{proof}

\section{Caloric polynomials}

It is a classical fact that $ \cP_d(\RR^n)$ consists of caloric polynomials, i.e., polynomials in $x,t$ that satisfy the heat equation (\cite{E1}, \cite{E2}, \cite{N}).  We   compute the dimensions of these spaces.

Given  a polynomial in $x$ and $t$,   define its {\emph{parabolic degree}} by considering  $t$ to have degree two.  Thus, $x_1^{m_1} x_2^{m_2} t^{m_0}$ has parabolic degree $m_1 + m_2 + 2m_0$.  A polynomial in $x,t$ is homogeneous if each monomial has the same parabolic degree.  Let $A_p^n$ denote the 
homogeneous   degree 
$p$  polynomials   on $\RR^n$.  The parabolic homogeneous degree $p$ polynomials $\cA^n_p$ are
\begin{align}
	\cA^n_p = A_p^n \oplus t\, A_{p-2}^n \oplus t^2 \, A_{p-4}^n \oplus \dots 
\end{align}

\begin{Lem}	\label{l:gotit}
For each positive integer $p$, we have    $\dim 	\left( \cP_p (\RR^n) \cap  \cA^n_p \right) = \dim A^n_p$ and 
\begin{align}
	\dim \, \cP_p (\RR^n) = \sum_{j=0}^{p} \, \dim A^n_j \, .
\end{align}
\end{Lem}

\begin{proof}
Observe that $\partial_t$ and $\Delta$ map $\cA^n_p$ to $\cA^n_{p-2}$.  Moreover, given any $u \in \cA^n_{p-2}$, we have 
\begin{align}
	(\partial_t - \Delta) \, \left[ t \, u - \frac{1}{2} \, t^2 (\partial_t - \Delta) u + \frac{1}{6} \, t^3  (\partial_t - \Delta)^2 u - \dots \right] = u \, .
\end{align}
Therefore, the map 
$
	(\partial_t - \Delta):\cA^n_p \to \cA^n_{p-2}
$
 is onto. Since the kernel of this map is $\cP_p (\RR^n) \cap  \cA^n_p $, we conclude that
 \begin{align}
 	\dim \left(  \cP_p (\RR^n) \cap  \cA^n_p   \right) = \dim \cA^n_p - \dim \cA^n_{p-2} =  \dim A^n_p \, .
 \end{align}
  This gives both  claims.
 \end{proof}

\begin{Lem}      \label{l:usingthis}
If $p \geq n$, then 
\begin{align}	\label{e:usingthis}
	\frac{1}{(n-1)!} \, p^{n-1} \leq \dim \, A_p^n   \leq  \frac{2^{n-1}}{(n-1)!} \, p^{n-1}  \, .
\end{align}
 \end{Lem}
 
 \begin{proof}
 To get the upper bound, we use that $p \geq n$ to get
 \begin{align}	 
	  \dim \,  A_p^n = \frac{ (p+n-1)!}{p! \, (n-1)!}  \leq   \frac{(p+n-1)^{n-1}}{(n-1)!}  \leq \frac{(2p)^{n-1}}{(n-1)!} = \frac{2^{n-1}}{(n-1)!} \, p^{n-1} \, .
\end{align}
The lower bound follows similarly since
$ 
	 \frac{ (p+n-1)!}{p! \, (n-1)!}  \geq   \frac{p^{n-1}}{(n-1)!}$.
	 \end{proof}

 The dimension bounds for $\cP_d (\RR^n)$ in \eqr{e:cpdrn}  follow  by combining Lemmas    \ref{l:gotit} and  \ref{l:usingthis}.

\subsection{Harmonic polynomials}

For each $j$, the Laplacian gives a linear map $\Delta: A^n_{j} \to A^n_{j-2}$.  The kernel   $H^n_j \subset A^n_j$ of this map is the linear space of 
 homogeneous harmonic polynomials of degree $j$ on $\RR^n$.  The next lemma shows that this map is onto:
 
 \begin{Lem}	\label{l:deltaonto}
 For each $d$, the map  $\Delta: A^n_{d+2} \to A^n_{d}$ is onto.
 \end{Lem}
 
 \begin{proof}
 Take an arbitrary $u \in A^n_{d}$.  For each nonnegative $\ell \leq d/2$, define $u_{\ell}$ and $v_{\ell}$ by
\begin{align}
	u_{\ell} &=   |x|^{2\ell} \, \Delta^{\ell} u \, , \\
	v_{\ell} & = |x|^2 \, u_{\ell} =    |x|^{2\ell+2} \, \Delta^{\ell} u \, .
\end{align}
Note that $u_0 = u$.  
 We will use repeatedly that if $v \in A^n_k$, then homogeneity gives
\begin{align}	\label{e:homogk}
	\langle x , \nabla v \rangle = k \, v \, .
\end{align}
Using this and $\Delta \, |x|^2 = 2n$, we get   for each $\ell$ that
\begin{align}
	\Delta \, v_{\ell} &= (\ell +1) \, (2n + 4 \ell) \, |x|^{2\ell} \, \Delta^{\ell} u + 2 \, \langle \nabla  |x|^{2(\ell +1)} , \nabla  \Delta^{\ell} u \rangle + 
		  |x|^{2(\ell +1)} \, \Delta^{\ell +1 } u \notag \\
		  &=  (\ell +1) \, (2n + 4 \ell) \, |x|^{2\ell} \, \Delta^{\ell} u + 4 \, (\ell +1) \, (d-2\ell)  \,   |x|^{2\ell } \,  \Delta^{\ell} u  + 
		  |x|^{2(\ell +1)} \, \Delta^{\ell +1 } u \\
		  &= (\ell +1) \, \left( 2n + 4d - 4\ell \right) \,  u_{\ell} + u_{\ell +1 } 
		  \notag 
		  \, .
\end{align}
Thus, if we define positive constants $c_{\ell} =  (\ell +1) \, \left( 2n + 4d - 4\ell \right)$, then we have that
\begin{align}	\label{e:uppertriangular}
	\Delta v_{\ell} = c_{\ell} \, u_{\ell} + u_{\ell +1} \, .
\end{align}
Let $k$ be the greatest integer less than or equal to $\frac{d}{2}$.  Note that $u_{k+1}=v_{k+1} \equiv 0$.
It follows from this and \eqr{e:uppertriangular} that
\begin{align}
	\Delta \, \left( v_k - c_{k} \, v_{k-1} + c_k \, c_{k-1} \, v_{k-2} - c_k \, c_{k-1} \, c_{k-2} \, v_{k-3} + \dots \right) 
\end{align}
is a nonzero multiple of $u_0 = u$, giving the lemma.
 \end{proof}

\begin{Cor}	\label{l:harmR}
For each positive integer $k$, we have  $\dim H^n_{k} = \dim A^n_{k} - \dim A^n_{k-2}$ and
\begin{align}	\label{e:tscmf}
	\dim \, \cH_{k} (\RR^n) = \dim A^n_{k} + \dim A^n_{k-1} \, .
\end{align}
\end{Cor}

\begin{proof}
Note that $\Delta: A^n_{j} \to A^n_{j-2}$ gives a linear map with kernel equal to $H^n_{j}$.  The map is onto by Lemma
\ref{l:deltaonto},  giving the first claim.  Summing the first claim gives \eqr{e:tscmf}.
\end{proof}

  \begin{Cor}	\label{c:c218}
  For each  $k$, \eqr{e:0p4} is an equality on $\RR^n$.
	  \end{Cor}
  
  \begin{proof}
  Corollary \ref{l:harmR}   and  Lemma \ref{l:gotit}  give
    \begin{align}	\label{e:hr1}
  	 \sum_{j=0}^k \dim \cH_{2j} (\RR^n) =   	 \sum_{j=0}^k  \left(   \dim A^n_{2j} + \dim A^n_{2j-1}    \right)  =  \sum_{i=0}^{2k} \dim  A^n_i = 	\dim \, \cP_{2k} (\RR^n) 
	 \, .
  \end{align}
  \end{proof}

\end{document}